\documentclass[11pt]{article}

\usepackage{amssymb,amsmath,amsfonts,amsthm}
\usepackage{latexsym}
\usepackage{graphics}
\usepackage{indentfirst}

\newtheorem{innercustomthm}{Theorem}
\newenvironment{customthm}[1]
  {\renewcommand\theinnercustomthm{#1}\innercustomthm}
  {\endinnercustomthm}

\newtheorem*{thm*}{Theorem}
\newtheorem{thm}{Theorem}
\newtheorem{lem}[thm]{Lemma}
\newtheorem{pro}[thm]{Proposition}
\newtheorem{obs}[thm]{Observation}
\newtheorem{cor}[thm]{Corollary}

\newtheorem{ques}[thm]{Question}
\newtheorem{proc}[thm]{Process}

\newcommand{\N}{\mathbb{N}}

\newcommand{\Ran}{\mathrm{Ran}}

\newcommand{\LL}{\mathcal{L}}

\begin{document}

\title{On the Equitable Choosability of the Disjoint Union of Stars}

\author{Hemanshu Kaul$^1$, Jeffrey A. Mudrock$^2$, and Tim Wagstrom$^3$}

\footnotetext[1]{Department of Applied Mathematics, Illinois Institute of Technology, Chicago, IL 60616. Email: {\tt {kaul@iit.edu}}}
\footnotetext[2]{Department of Mathematics, College of Lake County, Grayslake, IL 60030.  Email:  {\tt {jmudrock@clcillinois.edu}}}
\footnotetext[3]{Department of Mathematics and Applied Mathematical Sciences, University of Rhode Island, Kingston, RI 02881. Email: {\tt{twagstrom@uri.edu}}}

\date{}

\maketitle

\begin{abstract}


Equitable $k$-choosability is a list analogue of equitable $k$-coloring that was introduced by Kostochka, Pelsmajer, and West in 2003. It is known that if vertex disjoint graphs $G_1$ and $G_2$ are equitably $k$-choosable, the disjoint union of $G_1$ and $G_2$ may not be equitably $k$-choosable. Given any $m \in \N$ the values of $k$ for which $K_{1,m}$ is equitably $k$-choosable are known.  Also, a complete characterization of equitably $2$-choosable graphs is not known. With these facts in mind, we study the equitable choosability of $\sum_{i=1}^n K_{1,m_i}$, the disjoint union of $n$ stars.  We show that determining whether $\sum_{i=1}^n K_{1,m_i}$ is equitably choosable is NP-complete when the same list of two colors is assigned to every vertex. We completely determine when the disjoint union of two stars (or $n \geq 2$ identical stars) is equitably 2-choosable, and we present results on the equitable $k$-choosability of the disjoint union two stars for arbitrary $k$.


\medskip



\end{abstract}

\section{Introduction}\label{intro}

In this paper all graphs are nonempty, finite, simple graphs unless otherwise noted.  Generally speaking we follow~\cite{GJ79} and~\cite{W01} for terminology and notation.  The set of natural numbers is $\N = \{1,2,3, \ldots \}$.  For $m \in \N$, we write $[m]$ for the set $\{1, \ldots, m \}$.  If $G$ is a graph and $S \subseteq V(G)$, we use $G[S]$ for the subgraph of $G$ induced by $S$.  We write  $\Delta(G)$ for the maximum degree of a vertex in $G$.  We write $K_{n,m}$ for complete bipartite graphs with partite sets of size $n$ and $m$.  When $G_1$ and $G_2$ are vertex disjoint graphs, we write $G_1 + G_2$ or $\sum_{i=1}^2 G_i$ for the disjoint union of $G_1$ and $G_2$.  When $f$ is a function, we use $\Ran(f)$ to denote the range of $f$.

In this paper we study a list analogue of equitable coloring known as equitable choosability which was introduced in 2003 by Kostochka, Pelsmajer, and West~\cite{KP03}.  More specifically, we study the equitable choosability of the disjoint union of stars.  A \emph{star} is a complete bipartite graph with partite sets of size 1 and $m$ where $m \in \N$ (i.e., a copy of $K_{1,m}$).  We will occasionally need to consider complete bipartite graphs that are copies of $K_{1,0}$.  In such cases, we assume $K_{1,0} = K_1$ (i.e., a complete bipartite graph with partite sets of size 1 and 0 is a complete graph on one vertex).   We will now briefly review equitable coloring and list coloring.

\subsection{Equitable Coloring and List Coloring}

\subsubsection{Equitable Coloring}

Equitable coloring is a variation on the classical vertex coloring problem that began with a conjecture of Erd\H{o}s~\cite{E64} in~1964 which was proved in 1970 by Hajn\'{a}l and Szemer\'{e}di~\cite{HS70}.  In 1973 the notion of equitable coloring was formally introduced by Meyer~\cite{M73}.  A proper $k$-coloring $f$ of a graph $G$ is said to be an \emph{equitable $k$-coloring} if the $k$ color classes associated with $f$ differ in size by at most 1.  It is easy to see that for an equitable $k$-coloring, the color classes associated with the coloring are each of size $\lceil |V(G)|/k \rceil$ or $\lfloor |V(G)|/k \rfloor$.  We say that a graph $G$ is \emph{equitably $k$-colorable} if there exists an equitable $k$-coloring of $G$. Equitable colorings are useful when it is preferable to form a proper coloring without under-using or over-using any color (see~\cite{JR02,KJ06,P01,T73} for applications).

Unlike the typical vertex coloring problem, if a graph is equitably $k$-colorable, it need not be equitably $(k+1)$-colorable.  Indeed, $K_{2m+1,2m+1}$ is equitably $k$-colorable for each even $k$ less than $2m+1$, it is not equitably $(2m+1)$-colorable, and it is equitably $k$-colorable for each $k \geq 2m+2 = \Delta(K_{2m+1,2m+1})+1$ (see~\cite{LW96} for further details). In 1970, Hajn\'{a}l and Szemer\'{e}di~\cite{HS70} proved:
Every graph $G$ has an equitable $k$-coloring when $k \geq \Delta(G)+1$. In 1994, Chen, Lih, and Wu~\cite{CL94} conjectured that this result can be improved by 1 for most connected graphs by characterizing the extremal graphs as: $K_m$, $C_{2m+1}$, and $K_{2m+1,2m+1}$.  Their conjecture is still open and is known as the $\Delta$-Equitable Coloring Conjecture ($\Delta$-ECC for short).

Importantly, when it comes to the disjoint union of graphs, equitable $k$-colorings on components can be merged after appropriately permuting color classes within each component to obtain an equitable $k$-coloring of the whole graph.

\begin{thm}[\cite{YZ97}] \label{thm: disjoint}
Suppose $G_1$, $G_2$, \ldots, and $G_t$ are pairwise vertex disjoint graphs and $G = \sum_{i=1}^t G_i$.  If $G_i$ is equitably $k$-colorable for all $i \in [t]$, then $G$ is equitably $k$-colorable.
\end{thm}

On the other hand, an equitably $k$-colorable graph may have components that are not equitably $k$-colorable; for example, the disjoint union $G=K_{3,3}+K_{3,3}$ with $k=3$.  With this in mind, Kierstead and Kostochka~\cite{KK10} proposed an extension of the $\Delta$-ECC to the disjoint union of graphs.

\subsubsection{List Coloring}

List coloring is another variation on the classical vertex coloring problem introduced independently by Vizing~\cite{V76} and Erd\H{o}s, Rubin, and Taylor~\cite{ET79} in the 1970s.  For list coloring, we associate a \emph{list assignment} $L$ with a graph $G$ such that each vertex $v \in V(G)$ is assigned a list of colors $L(v)$ (we say $L$ is a list assignment for $G$).  The graph $G$ is \emph{$L$-colorable} if there exists a proper coloring $f$ of $G$ such that $f(v) \in L(v)$ for each $v \in V(G)$ (we refer to $f$ as a \emph{proper $L$-coloring} of $G$).  A list assignment $L$ is called a \emph{$k$-assignment} for $G$ if $|L(v)|=k$ for each $v \in V(G)$.  We say $G$ is \emph{k-choosable} if $G$ is $L$-colorable whenever $L$ is a $k$-assignment for $G$.

Suppose that $L$ is a list assignment for a graph $G$.  A \emph{partial $L$-coloring} of $G$ is a function $f: D \rightarrow \cup_{v \in V(G)} L(v)$ such that $D \subseteq V(G)$, $f(v) \in L(v)$ for each $v \in D$, and $f(u) \neq f(v)$ whenever $u$ and $v$ are adjacent in $G[D]$.  Also, the \emph{palette of colors associated with $L$} is $\cup_{v \in V(G)} L(v)$.  From this point forward, we use $\mathcal{L}$ to denote the palette of colors associated with $L$ whenever $L$ is a list assignment.  We say that $L$ is a \emph{constant $k$-assignment} for $G$ when $L$ is a $k$-assignment for $G$ and $|\mathcal{L}|=k$ (i.e., $L$ assigns the same list of $k$ colors to every vertex in $V(G)$).

\subsection{Equitable Choosability}

In 2003 Kostochka, Pelsmajer, and West~\cite{KP03} introduced a list analogue of equitable coloring called equitable choosability.  They used the word equitable to capture the idea that no color may be used excessively often.  If $L$ is a $k$-assignment for a graph $G$, a proper $L$-coloring of $G$ is an \emph{equitable $L$-coloring} of $G$ if each color in $\mathcal{L}$ appears on at most $\lceil |V(G)|/k \rceil$ vertices.  We call $G$ \emph{equitably $L$-colorable} when an equitable $L$-coloring of $G$ exists, and we say $G$ is \emph{equitably $k$-choosable} if an equitable $L$-coloring of $G$ exists for every $L$ that is a $k$-assignment for $G$.  It is conjectured in~\cite{KP03} that the Hajn\'{a}l-Szemer\'{e}di Theorem and the $\Delta$-ECC hold in the context of equitable choosability.



Much of the research on equitable choosability has been focused on  these conjectures.  There is not much research that considers the equitable $k$-choosability of a graph $G$ when $k < \Delta(G)$.  In~\cite{KP03} it is shown that if $G$ is a forest and $k \geq 1 + \Delta(G)/2$, then $G$ is equitably $k$-choosable.  It is also shown that this bound is tight for forests.  Also, in~\cite{KM18}, it is conjectured that if $T$ is a total graph, then $T$ is equitably $k$-choosable for each $k \geq \max \{\chi_\ell(T), \Delta(T)/2 + 2 \}$ where $\chi_\ell(T)$, the list chromatic number of $T$, is the smallest $m$ such that $T$ is $m$-choosable.  Finally, in~\cite{KM20}, it is remarked that determining precisely which graphs are equitably 2-choosable is open.

Furthermore, most results about equitable choosability state that some family of graphs is equitably $k$-choosable for all $k$ above some constant; even though, as with equitable coloring, if $G$ is equitably $k$-choosable, it need not be equitably $(k+1)$-choosable.  It is rare to have a result that determines whether a family of graphs is equitably $k$-choosable for each $k \in \N$.  Two examples of results of this form are: $K_{1,m}$ is equitably $k$-choosable if and only if $ m \leq \left\lceil (m+1)/k \right \rceil (k-1)$, and $K_{2,m}$ is equitably $k$-choosable if and only if $m \leq \left \lceil (m+2)/k \right \rceil (k-1)$ (see~\cite{MC19}).

It is important to note that the analogue of Theorem~\ref{thm: disjoint} does not hold in the setting of equitable choosability.  For example, we know $K_{1,6}$ and $K_{1,1}$ are equitably 3-choosable, but $K_{1,6} + K_{1,1}$ is not equitably 3-choosable.\footnote{One need only consider a constant 3-assignment to see this.} This fact along with the fact that the equitable choosability of $K_{1,m}$ has been completely characterized motivated us to study the following question which is the focus of this paper.

\begin{ques} \label{ques: main}
Suppose $n \geq 2$.  For which $k, m_1, \ldots, m_n \in \N$ is $\sum_{i=1}^n K_{1, m_i}$ equitably $k$-choosable?
\end{ques}

Since $\sum_{i=1}^n K_{1, m_i}$ is a forest, we know that it is equitably $k$-choosable whenever $k \geq 1 + \Delta(\sum_{i=1}^n K_{1, m_i})/2 = 1 + \max_{i \in [n]} m_i/2$. Even for this simple class of graphs, we do not know what happens when $k$ is smaller than $1 + \max_{i \in [n]} m_i/2$. In this paper, we make some further progress on Question~\ref{ques: main} in the case when $k=2$ and in the case when $n=2$.  We completely answer Question~\ref{ques: main} in the case when $n=k=2$. This can be seen as progress towards understanding which graphs are equitably $2$-choosable.

\subsection{Outline of the Paper and Open Questions}

We begin by studying Question~\ref{ques: main} in the case of equitable 2-choosability.  In Section~\ref{NP} we study the complexity of the decision problem STARS EQUITABLE 2-COLORING which is defined as follows.
\\
\\
\noindent \emph{Instance:}  An $n$-tuple $(m_1, \ldots, m_n)$ such that $m_i \in \N$ for each $i \in [n]$.
\\
\\
\noindent \emph{Question:} Is $\sum_{i=1}^n K_{1, m_i}$ equitably 2-colorable?\\

Perhaps surprisingly, since most coloring problems with 2 colors tend to be easy, we show that STARS EQUITABLE 2-COLORING is NP-complete. In studying when $\sum_{i=1}^n K_{1,m_i}$ is equitably 2-choosable, a possible natural starting point is to try to determine: for which $n$-tuples $(m_1, \ldots, m_n)$ is $\sum_{i=1}^n K_{1,m_i}$ not equitably 2-colorable and hence not equitably 2-choosable?  The fact that STARS EQUITABLE 2-COLORING is NP-complete tells us that this ``natural starting point" should not be pursued unless P $=$ NP.

STARS EQUITABLE 2-CHOOSABLITY is the decision problem whose instances are the same as STARS EQUITABLE 2-COLORING, but it asks the question: Is $\sum_{i=1}^n K_{1, m_i}$ equitably 2-choosable?  Clearly, this decision problem is closely related to Question~\ref{ques: main} in the case when $k=2$. The following question is open.

\begin{ques} \label{ques: 2choose}
Is STARS EQUITABLE 2-CHOOSABLITY NP-hard?
\end{ques}

In Section~\ref{2choose} we completely characterize when the disjoint union of 2 stars is equitably 2-choosable by proving the following.

\begin{thm}\label{thm: 2star}
Let $G= K_{1,m_1} + K_{1,m_2}$ where $1 \leq m_1 \leq m_2$. $G$ is equitably 2-choosable if and only if $m_2-m_1 \leq 1$ and $m_1+m_2 \leq 15$.
\end{thm}

Theorem~\ref{thm: 2star} makes progress on the task of identifying which graphs are equitably 2-choosable which in general is open (see~\cite{KM20}).  It is also worth noting that $K_{1,m_1} + K_{1,m_2}$ is equitably 2-colorable if and only if $|m_2-m_1| \leq 1$.  So, there are infinitely many equitably 2-colorable graphs that are the disjoint union of two stars, but there are only 14 equitably 2-choosable graphs (up to isomorphism) that are the disjoint union of two stars.  We end Section~\ref{2choose} by completely determining when the disjoint union of $n$ identical stars is equitably 2-choosable.

\begin{thm} \label{thm: same}
Suppose $n,m \in \N$, $n \geq 2$, and $G = \sum_{i=1}^n K_{1,m}$.  When $n$ is odd, $G$ is equitably 2-choosable if and only if $m \leq 2$.  When $n$ is even, $G$ is equitably 2-choosable if and only if $m \leq 7$.
\end{thm}

With these results in mind, the following open question is natural to ask.

\begin{ques} \label{ques: finite}
Suppose that $n$ is a fixed integer such that $n \geq 2$.  Are there only finitely many equitably 2-choosable graphs (up to isomorphism) that are the disjoint union of $n$ stars? 
\end{ques}

For a fixed integer $N$, $N$-STARS EQUITABLE 2-CHOOSABLITY is the decision problem whose instances are $N$-tuples of natural numbers of the form $(m_1, \ldots, m_N)$, and asks the question: Is $\sum_{i=1}^N K_{1, m_i}$ equitably 2-choosable?  If the answer to Question~\ref{ques: finite} is yes for a $n \geq 2$, then $N$-STARS EQUITABLE 2-CHOOSABLITY is not NP-hard for $N =n$ unless P$=$NP. Note that Theorem~\ref{thm: 2star} shows that this is true for $N=2$.

Finally, in Section~\ref{2stars} we study the equitable $k$-choosability of the disjoint union of two stars for arbitrary $k$. We use an extremal choice of a partial list coloring that minimizes the difference of the cardinalities of the sets of uncolored vertices in the two stars along with a greedy partial list coloring process to show the following.

\begin{thm}\label{thm: Main}
Let $k \in \N$, $1 \leq m_1 \leq m_2$, and $\rho=\lceil (m_1+m_2+2)/k\rceil$. If $m_2 \leq \rho(k-1)-1$ and $m_1+m_2 \leq 15 + \rho(k-2) $, then $K_{1,m_1} + K_{1,m_2}$ is equitably $k$-choosable.
\end{thm}

We also show that the converse of Theorem~\ref{thm: Main} does not hold. However, Theorem~\ref{thm: Main} is sharp in a sense. Lemma~\ref{lem: Only2} in Section~\ref{2choose} demonstrates that the first inequality in Theorem~\ref{thm: Main} is necessary for $K_{1,m_1} + K_{1,m_2}$ to be equitably $k$-choosable. However, this necessary condition alone is not sufficient for $K_{1,m_1} + K_{1,m_2}$ to be equitably $k$-choosable. Indeed Proposition~\ref{pro: tconstruct} in Section~\ref{2stars} implies that $K_{1, (k-1)(k^3-k+2)} + K_{1, k^3}$, which satisfies the first inequality but not the second inequality, is not equitably $k$-choosable whenever $k \geq 2$.  So, if one wishes to determine precisely when $K_{1,m_1} + K_{1,m_2}$ is equitably $k$-choosable for $k \geq 3$, the characterization needs to be stronger than $m_2 \leq \rho(k-1)-1$.  We suspect however that the second inequality in Theorem~\ref{thm: Main} can be relaxed quite a bit for $k \geq 3$.  This leads us to ask the following question which is a special case of Question~\ref{ques: main}.

\begin{ques} \label{ques: 2}
For which $k, m_1, m_2 \in \N$, is $K_{1,m_1} + K_{1,m_2}$ equitably $k$-choosable?
\end{ques}

\section{A Complexity Result} \label{NP}

To prove STARS EQUITABLE 2-COLORING is NP-complete, we will use the following well-known NP-complete problem~\cite{GJ79}: PARTITION which is defined as follows.
\\
\\
\noindent \emph{Instance:}  An $n$-tuple $(m_1, \ldots, m_n)$ such that $m_i \in \N$ for each $i \in [n]$.
\\
\\
\noindent \emph{Question:} Is there a partition $\{A,B\}$ of the set $[n]$ such that $\sum_{i \in A} m_i = \sum_{j \in B} m_j$?
\\

The following lemma captures the essence of why STARS EQUITABLE 2-COLORING is NP-complete.

\begin{lem} \label{lem: equiv}
Suppose that $n \in \N$ and $(m_1, \ldots, m_n)$ is an $n$-tuple such that $m_i \in \N$ for each $i \in [n]$ and $\sum_{i=1}^n m_i$ is even. There is a partition $\{A,B\}$ of the set $[n]$ such that $\sum_{i \in A} m_i = \sum_{j \in B} m_j$ if and only if $G=\sum_{i=1}^n K_{1,m_i+1}$ is equitably 2-colorable.
\end{lem}

\begin{proof}
Throughout this proof, we assume that for each $i \in [n]$ the bipartition of the copy of $K_{1, m_i+1}$ used to form $G$ is $A_i$, $B_i$ where $A_i$ is the partite set of size 1, and we suppose $s \in \N$ satisfies $2s = \sum_{i=1}^n m_i$.  Suppose that there is a partition $\{A,B\}$ of the set $[n]$ such that $\sum_{i \in A} m_i = \sum_{j \in B} m_j = s$.  Now, consider the proper 2-coloring $f$ of $G$ defined as follows.  For each $l \in A$ color each vertex in $A_l$ with 1, and color each vertex in $B_l$ with 2.  Similarly, for each $j \in B$ color each vertex in $A_j$ with 2, and color each vertex in $B_j$ with 1.  Now, it is easy to see
$$|f^{-1}(1)| = |A| + \sum_{j \in B} (m_j+1) = |A|+|B|+s = n+s.$$
Similarly, $|f^{-1}(2)|=n+s$.  So, $f$ is an equitable 2-coloring of $G$.

Conversely, suppose that $g: V(G) \rightarrow [2]$ is an equitable 2-coloring of $G$.  Since $|V(G)|=2n+2s$, we know that $|g^{-1}(1)|=|g^{-1}(2)|=n+s$.  We also have that for each $i \in [n]$, $f(B_i)$ is either $\{1\}$ or $\{2\}$.  Now, let
$$A = \{i \in [n] : f(B_i) = \{1\} \}$$
and $B = [n] - A$.  Clearly, $\{A,B\}$ is a partition of the set $[n]$.  Notice
$$n+s = |g^{-1}(1)| = \sum_{i \in A} |B_i| + \sum_{j \in B} |A_j|=|B|+\sum_{i \in A} (m_i+1) = n+\sum_{i \in A} m_i.$$
This implies that $\sum_{i \in A} m_i = s$.  A similar argument shows $\sum_{j \in B} m_j = s$ as desired.
\end{proof}

\begin{thm} \label{thm: complete}
STARS EQUITABLE 2-COLORING is NP-complete.
\end{thm}

\begin{proof}
We first show that STARS EQUITABLE 2-COLORING is in NP.  Suppose $\textbf{x} = (m_1, \ldots, m_n)$ such that $m_i \in \N$ for each $i \in [n]$ is an input that STARS EQUTIABLE 2-COLORING accepts.  Notice a proper 2-coloring of $\sum_{i=1}^n K_{1, m_i}$ can be represented by a binary string of length $n$ where the $i^{th}$ bit indicates which of two possible proper 2-colorings is used to color $K_{1, m_i}$.  Let $\textbf{y}$ be a certificate that represents an equitable 2-coloring of $\sum_{i=1}^n K_{1, m_i}$.  Clearly, the certificate is of size $n$ and \textbf{x} is of size $O(\sum_{i=1}^n (\lfloor \log_{2}(m_i) \rfloor + 1))$.  Finally, it is easy to verify $\textbf{y}$ represents an equitable 2-coloring in polynomial time.

Now, we will show that STARS EQUITABLE 2-COLORING is NP-Hard by showing there is a polynomial reduction from PARTITION to STARS EQUITABLE 2-COLORING.  Suppose $\textbf{x} = (m_1, \ldots, m_n)$ is an arbitrary $n$-tuple such that $m_i \in \N$ for each $i \in [n]$.  We view $\textbf{x}$ as an input into PARTITION.  If $\sum_{i=1}^n m_i$ is odd input $\textbf{y}=(3)$ into STARS EQUITABLE 2-COLORING; otherwise, input $\textbf{y} = (m_1+1, \ldots, m_n+1)$ into STARS EQUITABLE 2-COLORING.  Then, accept if and only if STARS EQUITABLE 2-COLORING accepts. It is obvious that this reduction runs in polynomial time.

We must show that there is a partition $\{A,B\}$ of the set $[n]$ such that $\sum_{i \in A} m_i = \sum_{j \in B} m_j$ if and only if there is an equitable 2-coloring of: $G= K_{1,3}$ in the case $\sum_{i=1}^n m_i$ is odd and $G = \sum_{i=1}^n K_{1,m_i+1}$ in the case $\sum_{i=1}^n m_i$ is even.  This statement clearly holds when $\sum_{i=1}^n m_i$ is odd, and the statement follows from Lemma~\ref{lem: equiv} when $\sum_{i=1}^n m_i$ is even.
\end{proof}

\section{Equitable 2-Choosability of the Disjoint Union of Stars} \label{2choose}

From this point forward, for any graph $G$ and $k \in \N$, we let $\rho(G,k) = \lceil |V(G)|/k \rceil$.  Additionally, when $G$ and $k$ are clear from context, we use $\rho$ to denote $\rho(G,k)$.  We begin this section with a lemma that gives us a simple necessary condition for the disjoint union of two stars to be equitably $k$-choosable. In this section, our primary use of this result will be in the case of equitable 2-choosability.

\begin{lem}\label{lem: Only}
Let $k \in \N$ and $G = K_{1,m_1} + K_{1,m_2}$ where $1 \leq m_1 \leq m_2$. If $m_2 > \rho (G,k)(k-1)-1 -\max\{0,m_1-\rho (G,k)+1\}$ then $G$ is not equitably $k$-choosable.
\end{lem}
\begin{proof}
Note that the result clearly holds when $k=1$.  So, we may assume that $k \geq 2$. Also note that $\rho \geq 1$, and the result clearly holds when $\rho = 1$.  So, we may assume that $\rho \geq 2$ (i.e., $k < m_1+m_2+2$).  Consider the $k$-assignment $L$ for $G$ given by $L(v) = [k]$ for all $v \in V(G)$. Let the bipartition of the copy of $K_{1,m_1}$ used to form $G$ be $\{w_0\}$, $A$ and the bipartition of the copy of $K_{1,m_2}$ used to form $G$ be $\{u_0\}$, $B$.  To prove the desired result, we will show that $G$ is not equitably $L$-colorable.

Suppose for the sake of contradiction that $f$ is an equitable $L$-coloring of $G$. Suppose that $f(w_0) = c_{w_0}$ and $f(u_0)=c_{u_0}$. We will derive a contradiction in the following two cases: (1) $c_{w_0}=c_{u_0}$ and (2) $c_{w_0} \neq c_{u_0}$. For the first case, since $f$ is proper, the vertices of $A \cup B$ are colored with colors from $[k] - \{c_{w_0}\}$. Note that
$$ m_1+m_2 = \sum_{i \in [k] - \{c_{w_0}\}} |f^{-1}(i)| \leq \rho (k-1)$$ which implies that $m_2 \leq \rho (k-1) -m_1$.  Since $\rho \geq 2$, $m_1 \geq m_1 + (2 - \rho) = 1 + m_1 - \rho+1$.  So, $-m_1 \leq -1-\max\{0,m_1-\rho+1\}$. Therefore we know that $m_2 \leq \rho(k-1) - 1 -\max\{0,m_1-\rho+1\}$ which is a contradiction.

In the second case, we know that the vertices of $A$ are colored with colors from $[k] - \{c_{w_0}\}$, and the vertices of $B$ are colored with colors from $[k] - \{c_{u_0}\}$.  Since $f$ is an equitable $L$-coloring it is clear that $|f^{-1}(c_{u_0}) \cap A| \leq \rho -1$ and $|f^{-1}(c_{w_0}) \cap B| \leq \rho -1$. Suppose $\max\{0, m_1 - \rho + 1\} = m_1 - \rho + 1$. We have that
$$m_1+m_2 = |f^{-1}(c_{w_0}) \cap B| + |f^{-1}(c_{u_0}) \cap A| + \sum_{i \in [k] - \{c_{w_0}, c_{u_0}\}} |f^{-1}(i)| \leq 2 (\rho -1) + \rho(k-2).$$
So, it follows that $m_2 \leq \rho (k-1) -1 - (m_1 - \rho + 1) = \rho (k-1) -1 -\max\{0, m_1 - \rho +1\}$ which is a contradiction. Now, suppose $\max\{0, m_1 - \rho + 1\} = 0$. Notice that $|f^{-1}(c_{w_0}) \cap A| \leq |A| = m_1$, which implies that
$$m_1+m_2 = |f^{-1}(c_{w_0}) \cap B| + |f^{-1}(c_{u_0}) \cap A| + \sum_{i \in [k] - \{c_{w_0}, c_{u_0}\}} |f^{-1}(i)| \leq (\rho -1) + m_1 + \rho(k-2).$$ Thus, we have that $m_2 \leq \rho (k - 1) - 1 = \rho (k - 1) - 1 - \max\{0, m_1 - \rho + 1\}$ which is a contradiction.
\end{proof}

Lemma~\ref{lem: Only} gives us a necessary condition for the disjoint union of two stars to be equitably $k$-choosable: if $G = K_{1,m_1} + K_{1,m_2}$ is equitably $k$-choosable, then $m_2 \leq \rho (G,k)(k-1)-1 -\max\{0,m_1-\rho (G,k)+1\}$.  Interestingly, $m_2 \leq \rho (G,k)(k-1)-1$ implies that $m_2 \leq \rho (G,k)(k-1)-1 -\max\{0,m_1-\rho (G,k)+1\}$.  So, we immediately have an equivalent necessary condition that is a bit easier to state.

\begin{lem} \label{lem: Only2}
Suppose $G = K_{1,m_1} + K_{1,m_2}$ where $1 \leq m_1 \leq m_2$.  If $m_2 \leq \rho (G,k)(k-1)-1$, then $m_2 \leq \rho (G,k)(k-1)-1 -\max\{0,m_1-\rho (G,k)+1\}$.  Consequently, the following two statements hold and are equivalent.
\\
(i) If $G$ is equitably $k$-choosable, then $m_2 \leq \rho (G,k)(k-1)-1 -\max\{0,m_1-\rho (G,k)+1\}$.
\\
(ii)  If $G$ is equitably $k$-choosable, then $m_2 \leq \rho (G,k)(k-1)-1$.~\footnote{It should be noted that while these statements are equivalent, the inequality in Statement~(i) holds with equality more often than the inequality in Statement~(ii).}  
\end{lem}  

\begin{proof}
Suppose for the sake of contradiction that $m_2 > \rho(k-1)-1 -\max\{0,m_1-\rho+1\}$.  We clearly get a contradiction when $0 \geq m_1 - \rho + 1$.  So, we suppose that $0 < m_1 - \rho + 1$.  Then, we have that $m_2 > \rho(k-1)-1 - m_1 + \rho - 1$ which implies that $m_1+m_2+2 > \rho k$ which is clearly a contradiction since $\rho = \lceil (m_1+m_2+2)/k \rceil$. 
\end{proof}

When we apply Lemma~\ref{lem: Only2} in this paper, we will always be using the Statement~(ii).  In the case of equitable 2-choosability, we may immediately deduce the following.

\begin{cor}\label{cor: Sub}
Let $G= K_{1,m_1}+K_{1,m_2}$ where $1 \leq m_1 \leq m_2$. If $m_2-m_1 \geq 2$, then $G$ is not equitably 2-choosable.
\end{cor}
\begin{proof}
Note that $\rho (G,2) \leq \lceil (m_2+m_2-2+2)/2\rceil = m_2$. So, $m_2 \geq \rho > \rho -1.$ Lemma~\ref{lem: Only2} implies $G$ is not equitably 2-choosable.
\end{proof}

We now present three lemmas that we will use to prove Theorem~\ref{thm: 2star}.

\begin{lem}\label{lem: Add}
Let $G= K_{1,m_1}+K_{1,m_2}$ where $1 \leq m_1 \leq m_2$. If $m_1+m_2 \geq 16$, then $G$ is not equitably 2-choosable.
\end{lem}
\begin{proof}
Note that by Corollary~\ref{cor: Sub} we may assume that $m_2 -m_1 \leq 1$. Let $G = G_1 + G_2$ where $G_1$ is a copy of $K_{1,m_1}$ and $G_2$ is a copy of $K_{1,m_2}$. It must be the case that $8 \leq m_1 \leq m_2 \leq m_1+1$. Now, suppose the bipartition of $G_1$ is $A'=\{w_0 \}$, $A=\{w_1, \ldots, w_{m_1} \}$ and the bipartition of $G_2$ is $B'=\{u_0 \}$, $B=\{u_1, \ldots, u_{m_2} \}$.

We will now construct a 2-assignment $L$ for $G$ for which there is no equitable $L$-coloring. Let $L(v) = [2]$ for $v \in B' \cup B$, $L(w_0) = \{3,4 \}$, $L(w_1)=L(w_2) = \{1,3\}$, $L(w_3)=L(w_4) = \{1,4\}$, $L(w_5)=L(w_6) = \{2,3\}$, $L(w_7)=L(w_8) = \{2,4\}$, and $L(w_i) = \{3,4\}$ for all $i \in [m_2] - [8]$. For the sake of contradiction, suppose that $G$ is equitably $L$-colorable. Let $f$ be an equitable $L$-coloring of $G$. Note $\rho (G,2) = \lceil (m_1+m_2+2)/2\rceil = m_2 + 1$. Clearly, $f(u_0)$ is either 1 or 2. Suppose $f(u_0) =2$.  Then it is clear that $f(u_i) = 1$ for all $i \in [m_2]$. Now, it is either the case that $f(w_0) =3$ or $f(w_0) = 4$. In the case that $f(w_0) = 3$ it must be that $f(w_1) = f(w_2) = 1$. However, this would imply that $|f^{-1}(1)| \geq m_2+2$ which is a contradiction. In the case that $f(w_0) = 4$ it must be that $f(w_3)=f(w_4) = 1$. However, this would also imply that $|f^{-1}(1)| \geq m_2+2$ which is a contradiction.

Suppose $f(u_0) = 1$.  Then it is clear that $f(u_i) = 2$ for all $i \in [m_2]$. Then we know that it is either the case that $f(w_0)= 3$ or $f(w_0) = 4$. In the case that $f(w_0) = 3$ it must be that $f(w_5)=f(w_6) = 2$ which would imply that $|f^{-1}(2)| \geq m_2+2$ a contradiction. Then in the case that $f(w_0) = 4$ it must be that $f(w_7) = f(w_8) = 2$ which implies that $|f^{-1}(2)| \geq m_2+2$ which is a contradiction.
\end{proof}

\begin{lem}\label{lem: Disjoint}
Let $G= G_1 + G_2$ where both $G_1$ and $G_2$ are copies of $K_{1,m}$ such that $m \in [7]$. Suppose the bipartition of $G_1$ is $\{w_0 \}$, $A= \{w_1, \ldots, w_{m} \}$, and the bipartition of $G_2$ is $\{u_0 \}$, $B=\{u_1, \ldots, u_{m} \}$.  If $L$ is a 2-assignment for $G$ such that $L(w_0) \cap L(u_0) = \emptyset$, then $G$ is equitably $L$-colorable.
\end{lem}
\begin{proof}
For the sake of contradiction, suppose there is a 2-assignment $K$ for $G$ such that $K(w_0) \cap K(u_0) = \emptyset$ and $G$ is not equitably $K$-colorable.  Among all such 2-assignments, choose a 2-assignment, $L$, with the smallest possible palette size.  Let $L(w_0) = \{k,c\}$ and $L(u_0) = \{t,d\}$. Clearly, $|\mathcal{L}| \geq 4$.  We will first show that $|\LL| > 4$.  

Assume $|\mathcal{L}|=4$; that is, $\mathcal {L} = \{t,k,c,d\}$. For each $\{c_1,c_2\}$ such that $|\{c_1,c_2\}| =2$ and $\{c_1,c_2\} \subseteq \{t,k,c,d\}$ let $a_{\{c_1,c_2\}} = |L^{-1}(\{c_1,c_2\}) \cap A|$ and $b_{\{c_1,c_2\}} = | L^{-1}(\{c_1,c_2\}) \cap B|$. We now consider all possible colorings of $w_0$ and $u_0$. For all $v \in A$ let $L^{(1)}(v) = L(v) - \{k\}$, and for all $v \in B$ let $L^{(1)}(v) = L(v)-\{t\}$. Since $G$ is not equitably $L$-colorable, it must be that among the lists $L^{(1)}(u_1),\ldots,L^{(1)}(u_m),L^{(1)}(w_1),\ldots, L^{(1)}(w_m)$ there are $m+2$ lists that are $\{d\}$ or there are $m+2$ that are $\{c\}$.  Suppose that $m+2$ of them are $\{d\}$ (the case where $m+2$ of them are $\{c\}$ is similar). Then, $a_{\{k,d\}} + b_{\{t,d\}} \geq m+2$ which implies $a_{\{k,d\}} \geq 2$.

Let $L^{(2)}(v) = L(v)-\{c\}$ for all $v \in A$, and let $L^{(2)}(v) = L(v) - \{t\}$ for all $v \in B$. It must be that among the lists $L^{(2)}(u_1),\ldots,L^{(2)}(u_m),L^{(2)}(w_1),\ldots, L^{(2)}(w_m)$ there are $m+2$ that are $\{d\}$ or there are $m+2$ that are $\{k\}$. Notice that if $m+2$ of these lists are $\{k\}$, $a_{\{c,k\}} + b_{\{t,k\}} \geq m+2$ which implies $a_{\{c,k\}} + b_{\{t,k\}} + a_{\{k,d\}} + b_{\{t,d\}} \geq 2m+4 > |V(G)|$ which is a contradiction.  So it must be that $m+2$ of these lists are $\{d\},$ and we have that $a_{\{c,d\}} + b_{\{t,d\}} \geq m+2$ which means $a_{\{c,d\}} \geq 2$.

Let $L^{(3)}(v) = L(v)-\{k\}$ for all $v \in A$, and let $L^{(3)}(v) = L(v) - \{d\}$ for all $v \in B$. It must be that among the lists $L^{(3)}(u_1),\ldots,L^{(3)}(u_m),L^{(3)}(w_1),\ldots,L^{(3)}(w_m)$ there are $m+2$ that are $\{c\}$ or there are $m+2$ that are $\{t\}$. Notice that if $m+2$ of these lists are $\{c\}$, $a_{\{c,k\}} + b_{\{c,d\}} \geq m+2$ which implies $a_{\{c,k\}} + b_{\{c,d\}} + a_{\{k,d\}} + b_{\{t,d\}} \geq 2m+4 > |V(G)|$ which is a contradiction. So it must be that $m+2$ of these lists are $\{t\}$, and we have that $a_{\{t,k\}} + b_{\{t,d\}} \geq m+2$ which means $a_{\{t,k\}} \geq 2$.

Let $L^{(4)}(v) = L(v)-\{c\}$ for all $v \in A$, and let $L^{(4)}(v) = L(v) - \{d\}$ for all $v \in B$. It must be that among the lists $L^{(4)}(u_1),\ldots,L^{(4)}(u_m),L^{(4)}(w_1),\ldots,L^{(4)}(w_m)$ there are $m+2$ that are $\{t\}$ or there are $m+2$ that are $\{k\}$. Notice that if $m+2$ of these lists are $\{k\}$, $a_{\{c,k\}} + b_{\{d,k\}} \geq m+2$ which implies $a_{\{c,k\}} + b_{\{k,d\}} + a_{\{k,d\}} + b_{\{t,d\}} \geq 2m+4 > |V(G)|$ which is a contradiction. So it must be that $m+2$ of these lists are $\{t\}$, and we have that $a_{\{t,c\}} + b_{\{t,d\}} \geq m+2$ which means $a_{\{c,t\}} \geq 2$.

So, $a_{\{c,t\}} +a_{\{t,k\}} + a_{\{c,d\}} + a_{\{k,d\}} \geq 8$. However this implies that $8 \leq |A| =m$ which is a contradiction.  

Now, we have that $|\mathcal{L}| \geq 5$.  For every $q \in \mathcal{L}-\{t,k,c,d\}$, let $\eta(q) = |\{v: q \in L(v)\}|$. In the case there is an $ r\in \LL -\{k,t,c,d\}$ satisfying $\eta (r) \leq m+1$, let $L'$ be a new 2-assignment for $G$ given by
$$L'(v) = \begin{cases}
      \{x,k\} &\text{if } L(v) = \{x,r\} \text{ for some } x \neq k \\
      \{k,c\} &\text{if } L(v) = \{k,r\} \\
      L(v) &\text{if } r \notin L(v)
   \end{cases}.$$
By the extremal choice of $L$, we know that $G$ is equitably $L'$-colorable, and we call such a coloring $f$. We then recolor all $v$ such that $f(v) \notin L(v)$ with $r$. Note that since $r \notin \{t,k,c,d\}$, this coloring is proper, and it is easy to see that it is also an equitable $L$-coloring of $G$.

Now, suppose that for every $q \in \mathcal{L}-\{t,k,c,d\}$ that $\eta (q) > m+1$.  Since $\mathcal{L}-\{t,k,c,d\}$ is nonempty, we may suppose that $s \in \mathcal{L}-\{t,k,c,d\}$.  Note that $|L^{-1} ( \{s,t\}) \cap B| + | L^{-1}(\{s,d\}) \cap B| \leq m$ and $|L^{-1}(\{s,k\}) \cap A| + |L^{-1}(\{s,c\}) \cap A| \leq m$.  Without loss of generality assume $|L^{-1}(\{s,t\}) \cap B| \leq m/2$ and $|L^{-1}(\{s,k\}) \cap A| \leq m/2$. Color all $v \in L^{-1}(\{s,t\}) \cap B$ and $v \in L^{-1}(\{s,k\}) \cap A$ with $s$. In doing this $s$ is used at most $m$ times. Then, arbitrarily color uncolored vertices that have $s$ in their lists with $s$ until exactly $m+1$ vertices are colored with $s$. Then, color $w_0$ with $k$ and $u_0$ with $t$. Now, let $U$ be the set of all uncolored vertices in $A \cup B$.  Let $L'(v) = L(v) - \{s,k\}$ for all $v \in U \cap A$, and let $L'(v) = L(v) - \{s,t\}$ for all $v \in U \cap B$. Clearly, $|U| = m-1$ and $|L'(v)| \geq 1$ for all $v \in U$. So, we can color each $v \in U$ with a color in $L'(v)$ to complete an equitable $L$-coloring of $G$.  This contradiction completes the proof.
\end{proof}

\begin{lem}\label{lem: Good}
Let $G= K_{1,m}+K_{1,m}$ where $m \in [7]$. Then $G$ is equitably 2-choosable.
\end{lem}

\begin{proof}
Suppose that the two components that make up $G$ are $G_1$ and $G_2$. We will show that $G$ is equitably 2-choosable by induction on $m$. The result holds when $m=1$ and when $m=2$ since $\Delta(G) \leq 2$ in these cases.  So, suppose that $2 < m \leq 7$ and the desired result holds for all natural numbers less than $m$.

Now, suppose the bipartition of $G_1$ is $\{w_0 \}$, $A= \{w_1, \ldots, w_{m} \}$ and the bipartition of $G_2$ is $\{u_0 \}$, $B=\{u_1, \ldots, u_{m} \}$. For the sake of contradiction, suppose there is a 2-assignment $L$ for $G$ such that $G$ is not equitably $L$-colorable. Let $G' = G-\{u_m,w_m\}$ and $K(v) = L(v)$ for all $v \in V(G')$.  By the inductive hypotheses there is an equitable $K$-coloring $f$ of $G'$ which uses no color more than $m$ times. The strategy of the proof is to now determine characteristics of $L$ and to then show that an equitable $L$-coloring of $G$ must exist. Let $L'(u_m) = L(u_m)-\{f(u_0)\}$ and $L'(w_m) = L(w_m) -\{f(w_0)\}$.

\emph{Observation 1}: \emph{$L'(u_m) = L'(w_m)$ and $|L'(u_m)| = 1$}. Suppose that $L'(u_m) \neq L'(w_m)$ or $|L'(u_m)| > 1$.  Notice it is possible to color $u_m$ and $w_m$ with two distinct colors from $L'(u_m)$ and $L'(w_m)$ respectively.  Combining this with $f$ completes an equitable $L$-coloring of $G$ which is a contradiction. 

So, we can assume $ L'(u_m) = L'(w_m) = \{c\}$.

\emph{Observation 2}: $|f^{-1}(c)| = m$. Suppose that $|f^{-1}(c)| < m$. Coloring $u_m$ and $w_m$ with $c$ and the other vertices in $G$ according to $f$ completes an equitable $L$-coloring of $G$ which is a contradiction.

Let $A' = A \cap f^{-1}(c)$ and $B' = B \cap f^{-1}(c)$. Without loss of generality assume that $A' = \{w_1,w_2,\ldots, w_a\}$ and $B' = \{u_1,u_2,\ldots, u_b\}$. Since $f(w_0) \neq c$ and $f(u_0)\neq c$, $a+b = m$. Without loss of generality assume $b \leq a$.  This implies $1 \leq b \leq a \leq m-1$ and $m/2 \leq a$.

\emph{Observation 3}: \emph{For all $v \in A' \cup \{w_m\}$, $L(v) = \{c,f(w_0)\}$, and for all $v \in B' \cup \{u_m\}$, $L(v)= \{c ,f(u_0)\}$.} Suppose that there is a $u_k \in B'$ such that $L(u_k) =\{x,c\}$ where $x \neq f(u_0)$. Since $|f^{-1}(f(u_0))| \geq 1$, $|f^{-1}(c)| = m$, and $|V(G')| = 2m$, we have that $|f^{-1}(x)| < m$.  Now, color all the vertices in $V(G')-\{u_k\}$ according to $f$, and color $u_k$ with $x$.  Coloring $u_m$ and $w_m$ with $c$ completes an equitable $L$-coloring which is a contradiction.  A similar argument can be used to show that for all $v \in A' \cup \{w_m\}$, $L(v) = \{c,f(w_0)\}$.

Now suppose that $L(u_0) = \{f(u_0) , d\}$ and $L(w_0) = \{ f(w_0), l\}$.

\emph{Observation 4}: $d \neq c$. Suppose that $d = c$. Color all the vertices in $V(G')$ according to $f$.  Then, recolor $u_0$ with $c$ and each vertex in $B'$ with $f(u_0)$. Finally, color $u_m$ with $f(u_0)$ and $w_m$ with $c$. Note that the number of times $c$ is used is exactly $a+1+1 \leq m+1$. Also note that the number of times $f(u_0)$ is used is at most $b+1 + \max \{1, m-(a+1) \} \leq m+1$.  So, we have constructed an equitable $L$-coloring of $G$ which is a contradiction.

\emph{Observation 5}: $l = c$. Suppose that $l \neq c$. Color all the vertices in $V(G')$ according to $f$.  Then, recolor $w_0$ with $l$.  Also, color $w_m$ with $f(w_0)$ and $u_m$ with $c$. For each $v \in (A -(A'\cup\{w_m\}))$ such that $f(v) = l$, we recolor $v$ with the element in $L(v) -\{l\}$. Let $r = |\{v \in A-(A'\cup \{w_m\}): f(v)=l\}|$, and note that $0 \leq r \leq m -(a+1) \leq m/2 \leq a$. At this stage, we know that $c$ is used $m+1+z$ times where $z$ is an integer satisfying $0 \leq z \leq r \leq a$. If $z \geq 1$, recolor $w_1, \ldots, w_z$ with $f(w_0)$. Note that the resulting coloring is proper. Moreover, the resulting coloring uses $c$ exactly $m+1$ times.  So, it must be an equitable $L$-coloring of $G$ since $|V(G)| = 2m+2$. This is a contradiction.

We now have that $L(u_0) = \{f(u_0), d\}$, $L(w_0) = \{f(w_0), c\}$, and $c \neq d$.

\emph{Observation 6}: $f(u_0) \neq f(w_0)$. Suppose that $f(u_0) = f(w_0)$. Color all the vertices in $V(G')$ according to $f$.  Then, recolor $w_0$ with $c$, and for each $v\in A'$, recolor $v$ with $f(w_0)$. Finally, color $w_m$ with $f(w_0)$ and $u_m$ with $c$. Note that since no vertices in $B$ are colored with $f(w_0)$, it must be that $f(w_0)$ is used at most $m+1$ times. Also note that $c$ is used at most $b+1+1 \leq m+1$ times. So, we have constructed an equitable $L$-coloring of $G$ which is a contradiction.

\emph{Observation 7}: $f(w_0) \neq d$. Suppose that $f(w_0) =d$. Color all the vertices in $V(G')$ according to $f$.  Then, recolor $w_0$ with $c$ and $u_0$ with $d$. For all $v \in A'$, recolor $v$ with $d$.  Also, recolor all $v \in B -(B'\cup \{u_m\})$ satisfying $f(v) = d$ with the element in $L(v) - \{d\}$. Also, color $w_m$ with $d$ and $u_m$ with $c$.  Our resulting coloring is clearly proper. Note that $d$ is used exactly $a+2$ times which means it is used at most $m+1$ times. Also note that $c$ is used at most $m+1$ times and at least $b+2$ times. Thus $c$ and $d$ are used at least $a+b+4 = m+4 > m+1 = |V(G)|/2$ times. Thus, we have constructed an equitable $L$-coloring of $G$ which is a contradiction.

Observations 6 and 7 allow us to conclude $L(w_0) \cap L(u_0) = \emptyset$ by which Lemma~\ref{lem: Disjoint} implies there exists an equitable $L$-coloring of $G$ which is a contradiction.
\end{proof}

We are now ready to prove Theorem~\ref{thm: 2star} which we restate.

\begin{customthm} {\ref{thm: 2star}}
Let $G= K_{1,m_1} + K_{1,m_2}$ where $1 \leq m_1 \leq m_2$. $G$ is equitably 2-choosable if and only if $m_2-m_1 \leq 1$ and $m_1+m_2 \leq 15$.
\end{customthm}

\begin{proof}
We begin by assuming that $m_2-m_1 \geq 2$ or $m_1+m_2 \geq 16$. By Corollary~\ref{cor: Sub} and Lemma~\ref{lem: Add} we know that in both cases $G$ is not equitably 2-choosable.

Conversely, suppose that $m_1+m_2 \leq 15$ and $m_2-m_1 \leq 1$. In the case that $m_1 = m_2$ we know that the desired result holds by Lemma~\ref{lem: Good}. So we may assume that $m_2 = m_1+1$. Suppose that $L$ is an arbitrary 2-assignment for $G$, and let $m=m_1$. Let the copies of $K_{1,m}$ and $K_{1,m+1}$ that make up $G$ be $G_1$ and $G_2$ respectively. Suppose the bipartition of $G_1$ is $\{w_0 \}$, $A= \{w_1, \ldots, w_{m} \}$ and the bipartition of $G_2$ is $\{u_0 \}$, $B=\{u_1, \ldots, u_{m+1} \}$. Let $G' = G - \{u_{m+1}\}$, and let $L'(v) = L(v)$ for all $v \in V(G')$. We know by Lemma~\ref{lem: Good} that there exists an equitable $L'$-coloring $f$ of $G'$.  Suppose the vertices in $V(G')$ are colored according to $f$.  Note that $\rho(G',2) < \rho(G,2)$. Thus, we can complete an equitable $L$-coloring of $G$ by coloring $u_{m+1}$ with a color in $L(u_{m+1}) - \{f(u_0) \}$.
\end{proof}

We end this section by proving Theorem~\ref{thm: same} which we restate.  It should be noted that when $G=\sum_{i=1}^n K_{1,m}$, $G$ is a forest of maximum degree $m$, and we know that $G$ is equitably 2-choosable when $m \in [2]$ (see Section~\ref{intro}).

\begin{customthm} {\ref{thm: same}}
Suppose $n,m \in \N$, $n \geq 2$, and $G = \sum_{i=1}^n K_{1,m}$.  When $n$ is odd, $G$ is equitably 2-choosable if and only if $m \leq 2$.  When $n$ is even, $G$ is equitably 2-choosable if and only if $m \leq 7$.
\end{customthm}

\begin{proof}
Throughout this argument, let $G_1,G_2,\ldots, G_n$ be the components of $G$.  Let $A'_i = \{w_{i,0}\}$ and $A_i = \{w_{i,1},\ldots,w_{i,m}\}$ be the bipartition of $G_i$ for each $i \in [n]$.

First, suppose that $n$ is odd.  We know that $G$ is equitably 2-choosable when $m\leq 2$.  For the converse, we suppose that $m \geq 3$, and we will construct a 2-assignment $L$ for $G$ for which there is no equitable $L$-coloring.  Let $L(v) = [2]$ for all $v \in V(G)$.  For the sake of contradiction, suppose that $G$ is equitably $L$-colorable.  Let $f$ be an equitable $L$-coloring of $G$.  Note that $$\max\{|f^{-1}(1)|, |f^{-1}(2)|\} \geq m \left\lceil \frac{n}{2} \right\rceil + \left\lfloor \frac{n}{2} \right\rfloor = \frac{nm+m+n-1}{2}.$$ Also note that $\rho(G,2)=\lceil(n(m+1))/2\rceil \leq (n(m+1) + 1)/2 = (nm+n+1)/2$. Since $m \geq 3$ we know that $m-1 > 1$.  Therefore, we see that $\max\{|f^{-1}(1)|, |f^{-1}(2)|\} > \rho$ which is a contradiction.

Now, suppose that $n$ is even.  We begin by assuming that $m \geq 8$.  We will now construct a 2-assignment $L$ for $G$ for which there is no equitable $L$-coloring.  Let $L(v) = [2]$ for all $v \in \bigcup_{i=2}^n(A'_i \cup A_i)$, $L(w_{1,0}) = \{3,4\}$, $L(w_{1,1})=L(w_{1,2}) = \{1,3\}$, $L(w_{1,3})=L(w_{1,4}) = \{1,4\}$, $L(w_{1,5})=L(w_{1,6}) = \{2,3\}$, $L(w_{1,7}) = L(w_{1,8}) = \{2,4\}$, and $L(v) = \{3,4\}$ for all $v \in A_1 - \{w_{1,1},w_{1,2},\ldots,w_{1,8}\}$.  For the sake of contradiction, suppose $G$ is equitably $L$-colorable, and suppose $f$ is an equitable $L$-coloring of $G$. Note that $\rho (G,2) = \lceil n(m+1)/2\rceil = n(m+1)/2$. We calculate
\begin{align*}
\max\{|f^{-1}(1)|, |f^{-1}(2)|\} \geq m \left\lceil \frac{n-1}{2} \right\rceil + \left\lfloor \frac{n-1}{2} \right\rfloor + 2 = \frac{mn}{2} + \frac{n-2}{2} +2 = \frac{mn+n+2}{2}.
\end{align*}
It is easy to see that $\max\{|f^{-1}(1)|, |f^{-1}(2)|\} > \rho$ which is a contradiction. Thus, $G$ is not equitably 2-choosable.  Conversely, suppose that $m \leq 7$. Let $G^{(i)} = G_{2i}+G_{2i-1}$ for all $i \in [n/2]$. Suppose that $L$ is an arbitrary $2$-assignment for $G$. Let $L^{(i)}$ be $L$ restricted to the vertices of $G^{(i)}$. By Theorem~\ref{thm: 2star} we know that there is a proper $L^{(i)}$-coloring $f_i$ of $G^{(i)}$ that uses no color more than $m+1$ times for each $i \in [n/2]$. Coloring the vertices of $G$ according to $f_1, \ldots, f_{n/2}$ achieves an equitable $L$-coloring of $G$ since no color could possibly be used more than $(m+1)n/2 = \rho$ times.
\end{proof}

\section{Equitable Choosability of the Disjoint Union of Two Stars} \label{2stars}

We begin by stating an inductive process that will be used throughout the remainder of the paper. Note here $\epsilon$ is used to indicate `equitable'.

\begin{proc}\label{proc: greedy}
\textbf{$\epsilon$-greedy process}: The $\epsilon$-greedy process takes as input: a graph $G = G_1 + G_2$ where $G_i$ is a copy of $K_{1,m_i}$ for $ i \in [2]$, and a $k$-assignment $L$ where $k \geq 3$. It outputs $G_\epsilon$ where $G_\epsilon$ is an induced subgraph of $G$, a list assignment $L_\epsilon$ for $G_\epsilon$, and a partial $L$-coloring $g_\epsilon$ of $G$ that colors the vertices in $V(G)-V(G_\epsilon)$.

Suppose the bipartition of $G_1$ is $\{w_0 \}$, $A= \{w_1, \ldots, w_{m_1} \}$ and the bipartition of $G_2$ is $\{u_0 \}$, $B=\{u_1, \ldots, u_{m_2} \}$. To begin we determine whether there is a color that appears in at least $\rho(G,k)$ of the lists associated with the vertices in $A \cup B$. If no such color exists let $G_\epsilon =G$, $L_\epsilon = L$, and $g_\epsilon$ be a function with an empty domain, then the process terminates. Otherwise there exists a color $c_1$ that is in at least $\rho$ of the lists associated with the vertices in $A \cup B$, and we arbitrarily put $\rho$ of these vertices in a set $C_1$.  We consider this the first step of the $\epsilon$-greedy process.

If $k=3$ let $G_\epsilon = G - C_1$, $L_\epsilon$ be the list assignment for $G_\epsilon$ given by: $L_\epsilon (v) = L(v) - \{c_1\}$ for all $v \in V(G_\epsilon)$, and $g_\epsilon: C_1 \rightarrow \{c_1\}$ be the partial $L$-coloring of $G$ given by $g_\epsilon(v) =c_1$ whenever $v \in C_1$, then the process terminates.

If $k \geq 4$, we proceed inductively. For each $t= 2, \ldots , k-2$ if the process has not terminated in the $(t-1)$th step we determine whether there is a color in $\mathcal{L}-\{c_1, \ldots , c_{t-1}\}$ that appears in at least $\rho$ of the lists associated with the vertices in $(A \cup B)- \bigcup_{i=1}^{t-1}C_i$. If no such color exists let $G_\epsilon = G - \bigcup_{i=1}^{t-1} C_i$, $L_\epsilon$ be the list assignment for $G_\epsilon$ given by: $L_\epsilon (v) = L(v) - \{c_i: i \in [t-1]\}$ for all $v \in V(G_\epsilon)$, and $g_\epsilon: \bigcup_{i=1}^{t-1}C_i \rightarrow \{c_i: i \in [t-1]\}$ be the partial $L$-coloring of $G$ given by $g_\epsilon(v) =c_i$ whenever $v \in C_i$, then the process terminates. Otherwise there exists a color $c_t \in \mathcal{L}-\{c_1, \ldots , c_{t-1}\}$ that is in at least $\rho$ of the lists associated with the vertices in $(A \cup B)- \bigcup_{i=1}^{t-1}C_i$, and we arbitrarily put $\rho$ of these vertices in a set $C_t$.

If the process does not terminate when $t = k-2$ let $G_\epsilon = G - \bigcup_{i=1}^{k-2} C_i$, $L_\epsilon$ be the list assignment for $G_\epsilon$ given by: $L_\epsilon (v) = L(v) - \{c_i: i \in [k-2]\}$ for all $v \in V(G_\epsilon)$, and $g_\epsilon: \bigcup_{i=1}^{k-2}C_i \rightarrow \{c_i: i \in [k-2]\}$ be the partial $L$-coloring of $G$ given by $g_\epsilon(v) =c_i$ whenever $v \in C_i$, then the process terminates.
\end{proc}

By the definition of the $\epsilon$-greedy process we easily obtain the following observation and two lemmas.

\begin{obs}\label{obs: obs2}
Suppose that the $\epsilon$-greedy process is run on $G = K_{1,m_1}+K_{1,m_2}$ with a $k$-assignment $L$ where $k \geq 3$. Then $|L_\epsilon (v)| \geq 2$ for all $v \in V(G_\epsilon)$.
\end{obs}

\begin{lem}\label{lem: thing1}
Suppose that the $\epsilon$-greedy process is run on $G = K_{1,m_1} + K_{1,m_2}$ with a $k$-assignment $L$. Let $\{w_0 \}$ (resp., $\{u_0 \}$) denote the partite set of size one for the copy of $K_{1,m_1}$ (resp., $K_{1,m_2}$) used to form $G$. If no color appears in at least $\rho (G,k)$ of the lists associated with the vertices in $V(G_\epsilon) - \{w_0,u_0\}$ by $L_\epsilon$, then $G$ is equitably $L$-colorable.
\end{lem}

\begin{proof}
Note that since no color appears in at least $\rho (G,k)$ of the lists associated with the vertices in $V(G_\epsilon) - \{w_0,u_0\}$ by $L_\epsilon$, any proper $L_\epsilon$-coloring cannot possibly use a color more than $\rho (G,k)$ times. By Observation~\ref{obs: obs2} and the fact that $G_\epsilon$ is 2-choosable there must exist a proper $L_\epsilon$-coloring of $G_\epsilon$. Such a coloring combined with $g_\epsilon$ yields an equitable $L$-coloring of $G$.
\end{proof}

\begin{lem}\label{lem: thing2}
Suppose that the $\epsilon$-greedy process is run on $G = K_{1,m_1} + K_{1,m_2}$ with a $k$-assignment $L$. Let $\{w_0 \}$ (resp., $\{u_0 \}$) denote the partite set of size one for the copy of $K_{1,m_1}$ (resp., $K_{1,m_2}$) used to form $G$.  If there is a color that appears in at least $\rho (G,k)$ of the lists associated with the vertices in $V(G_\epsilon) - \{w_0,u_0\}$ by $L_\epsilon$, then $|\Ran(g_\epsilon)| = k-2$.
\end{lem}

\begin{proof}
We prove the contrapositive.  Suppose that $|\text{Ran}(g_\epsilon)| \neq k-2$. It is easy to verify that $|\text{Ran}(g_\epsilon)| < k-2$. For the sake of contradiction, suppose that there exists a color that appears in at least $\rho (G,k)$ of the lists associated with the vertices $V(G_\epsilon) - \{w_0,u_0\}$ by $L_\epsilon$. This implies that the $\epsilon$-greedy process would have been able to continue to the $(|\text{Ran}(g_\epsilon)| +1)$th step, a contradiction.
\end{proof}

Before proving Theorem~\ref{thm: Main} we prove the following Lemma.

\begin{lem} \label{lem: reduce}
Suppose that $0 \leq m_1$ and $m_2 \geq \max \{2, m_1 \}$. Let $G = K_{1,m_1} + K_{1,m_2}$ with $A$ (resp., $B$) denoting the partite set of size $m_1$ (resp., $m_2$) in the copy of $K_{1,m_1}$ (resp., $K_{1,m_2}$) used to form $G$.  Suppose $L$ is a list assignment for $G$ such that: $|L(b)| \geq 3$ for all $b \in B$, $|L(a)| \geq 2$ for all $a \in V(G)-B$, and there is a color $c$ that appears in at least $\lfloor (m_1+m_2+2)/2 \rfloor$ of the lists associated with the vertices in $A \cup B$. Then, there exists a proper $L$-coloring of $G$ that uses no color more than $\lfloor (m_1+m_2+2)/2 \rfloor$ times.
\end{lem}

\begin{proof}
Let $\sigma = \lfloor (m_1+m_2+2)/2 \rfloor$ and $C = \{v \in A\cup B: c \in L(v)\}$. We begin by coloring all the vertices in $C \cap A$ with $c$, and note that less than $\sigma$ vertices are colored in doing this since $|A| = m_1 < \sigma (G,k)$. We arbitrarily color $\sigma - |C \cap A|$ vertices in $B \cap C$ with $c$.  Let $C_1$ be the set of vertices colored with $c$, and let $G' = G - C_1$. Let $L'(v) = L(v) - \{c\}$ for all $v \in V(G')$. Let $\{w_0\}$ (resp., $\{u_0\}$) be the partite set of size 1 in the copy of $K_{1,m_1}$ (resp., $K_{1,m_2}$) used to form $G$. Note that $|L'(w_0)| \geq 1$, $|L'(u_0)| \geq 1$, and $|L'(v)|\geq 2$ for all $v \in V(G') - \{w_0,u_0\}$.  Now, order the vertices of $G'$ in such a way that $u_0$ and $w_0$ are the first two.  Then, greedily color the vertices so that a proper $L'$-coloring of $G'$ is achieved. Note that the resulting proper $L'$-coloring either uses at least 2 colors or $|V(G')|=2$.  Consequently, the resulting proper $L'$-coloring uses no color more than $\sigma$ times (note that $m_2 \geq 2$ implies that $\sigma \geq 2$ for the case in which $|V(G')| = 2$). It follows that this proper $L'$-coloring of $G'$ completes a proper $L$-coloring of $G$ with the desired property.
\end{proof}

We are now ready to prove Theorem~\ref{thm: Main} which we will restate.

\begin{customthm} {\ref{thm: Main}}
Let $k \in \N$, $1 \leq m_1 \leq m_2$, and $\rho=\lceil (m_1+m_2+2)/k\rceil$. If $m_2 \leq \rho(k-1)-1$ and $m_1+m_2 \leq 15 + \rho(k-2) $ then $K_{1,m_1} + K_{1,m_2}$ is equitably $k$-choosable.
\end{customthm}

\begin{proof}
Let $G= K_{1,m_1} + K_{1,m_2}$.  Suppose the bipartition of the copy of $K_{1,m_1}$ used to form $G$ is $\{w_0 \}$, $A= \{w_1, \ldots, w_{m_1} \}$, and suppose the bipartition of the copy of $K_{1,m_2}$ used to form $G$ is $\{u_0 \}$, $B=\{u_1, \ldots, u_{m_2} \}$. The result is obvious when $k=1$, and it follows from Theorem~\ref{thm: 2star} when $k=2$.  The result is also obvious when $\rho= 1$.  So, we suppose that $k \geq 3$ and $\rho \geq 2$.

Suppose $L$ is an arbitrary $k$-assignment of $G$.  We must show an equitable $L$-coloring of $G$ exists.  Let $\mathcal{C}$ denote the set of all partial $L$-colorings $f: D \rightarrow \mathcal{L}$ of $G$ such that $D \subset A \cup B$, $|D| = \rho (k-2)$, $|f(D)| = k-2$, and each color class associated with $f$ is of size $\rho $. Notice that elements of $\mathcal{C}$ need not have the same domain. Suppose that we run the $\epsilon$-greedy process on $G$ and $L$. If there is no color that appears in at least $\rho $ of the lists associated with the vertices in $V(G_\epsilon) - \{w_0,u_0\}$ by $L_\epsilon$, we know by Lemma~\ref{lem: thing1} that $G$ is equitably $L$-colorable. So we may assume that there exists a color that appears in at least $\rho$ of the lists associated with the vertices in $V_\epsilon - \{w_0,u_0\}$ by $L_\epsilon$. By Lemma~\ref{lem: thing2} we know that $|\text{Ran}(g_\epsilon)| = k-2$. It is then easy to verify that $g_\epsilon \in \mathcal{C}$.  For each $f \in \mathcal{C}$ let $U_A^f$ (resp., $U_B^f$) be the set of vertices in $A$ (resp., $B$) not colored by $f$.

Among all elements of $\mathcal{C}$ we choose a function $g: D' \rightarrow \mathcal{L}$ such that $||U_A^g| -|U_B^g||$ is as small as possible. Let $\mu_A = |U_A^g|$, $\mu_B = |U_B^g|$, and $G' = G -D'$. Note that it is possible that either $\mu_A = 0$ or $\mu_B = 0$. Also note that $G'$ is a copy of $K_{1,\mu_A} + K_{1,\mu_B}$. If $\mu_A = 0$ (resp., $\mu_B = 0$) then $G'$ would be a copy $K_1 + K_{1,\mu_B}$ (resp., $K_1 + K_{1,\mu_A}$). Let $L'$ be the list assignment for $G'$ defined as follows: $L'(v) = L(v) - g(D')$ for all $v \in V(G')$. Note that $|L'(v)| \geq 2$ for all $v \in V(G')$.  Suppose that $U_A^g = \{a_1, \ldots, a_{\mu_A} \}$ and $U_B^g = \{b_1, \ldots, b_{\mu_B} \}$. Note that there must be a color that appears in at least $\rho$ of the lists assigned by $L'$ to the vertices in $U_A^g \cup  U_B^g$, for if this was not so we could complete an equitable $L$-coloring of $G$ through a similar approach to that of the proof of Lemma~\ref{lem: thing1}.

We now show that an equitable $L$-coloring of $G$ exists in each of the following three cases:
\begin{enumerate}
\item $|\mu_A - \mu_B| \leq 1$;
\item $\mu_B - \mu_A \geq 2$ and $U_A^g \neq A$, or $\mu_A - \mu_B \geq 2$;
\item $\mu_B - \mu_A \geq 2$ and $U_A^g = A$.
\end{enumerate}

For case one notice that $\mu_A$ and $\mu_B$ are positive since $\rho \geq 2$.  Now, for each $v \in V(G')$ such that $|L'(v)| > 2$ we arbitrarily delete colors from $L'(v)$ until it is of size 2. After this is complete, $L'$ is a 2-assignment for $G'$. Note that $\mu_A + \mu_B = m_1+m_2-\rho (k-2) \leq 15$. Theorem~\ref{thm: 2star} implies that an equitable $L'$-coloring $h$ of $G'$ exists. Since $|V(G)| \leq k \rho $, we know that $|V(G')| \leq 2 \rho $. So $h$ uses no color more than $\rho $ times. Combining $h$ and $g$ completes an equitable $L$-coloring of $G$.

For the second case first suppose that $\mu_B - \mu_A \geq 2$ and $U_A^g \neq A$. Clearly $\mu_B \geq \max\{2, \mu_A \}$. We claim that for each $b_i \in U_B^g$, $|L'(b_i)| \geq 3$.  To see why this is so, suppose there is some $b_j \in U_B^g$ such that $|L'(b_j)|=2$.  Then, $g(D') \subset L(b_j)$.  Since $U_A^g \neq A$, there is a $w \in A - U_A^g$.  Now, we can construct an element $h$ of $\mathcal{C}$ from $g$ by removing the color $g(w)$ from vertex $w$ and coloring $b_j$ with $g(w)$.  Then, $|U_B^h| - |U_A^h| < \mu_B - \mu_A$ which is a contradiction to the minimality of $|\mu_B - \mu_A|$.  So, for each $b_i \in U_B^g$, $|L'(b_i)| \geq 3$.  Since $(\mu_A + \mu_B + 2)/2 =   |V(G')|/2 \leq \rho $ and there is a color in at least $\rho $ of the lists assigned by $L'$ to the vertices in $U_A^g \cup U_B^g$, Lemma~\ref{lem: reduce} implies that there is a proper $L'$-coloring of $G'$ that uses no color more than $\rho $ times.  Combining such a coloring with $g$ completes an equitable $L$-coloring of $G$.

If instead we have that $\mu_A - \mu_B \geq 2$, we claim that $U_B^g \neq B$.  To see why this is so, note that if $U_B^g = B$, then we have that $m_2 = |B| = \mu_B \leq \mu_A - 2 < |A|=m_1$ which is a contradiction.  Since $U_B^g \neq B$ an argument similar to the argument employed at the start of the second case can be used to show that there is an equitable $L$-coloring of $G$.

Finally, we turn our attention to case three, and we suppose that $\mu_B - \mu_A \geq 2$ and $U_A^g = A$. Note that clearly $\mu_B \geq \mu_A > 0$. In this case we let $d = \mu_B - \mu_A$.  Also, as in case one, for each $v \in V(G')$ such that $|L'(v)| > 2$ we arbitrarily delete colors from $L'(v)$ until it is of size 2 so that $L'$ becomes a 2-assignment for $G'$.  By the given bound on $m_2$ and the fact that $U_A^g = A$, we know that:
$$\mu_B = m_2 - \rho (k-2) \leq \rho (k-1) - 1 - \rho (k-2) = \rho  - 1.$$
Now, let $G'' = G' - \{b_1, \ldots, b_d \}$, and let $L''$ be the 2-assignment for $G''$ obtained by restricting the domain of $L'$ to $V(G'')$.  Clearly, $G''$ is a copy of $K_{1, \mu_A} + K_{1, \mu_A}$.  Also,  $2\mu_A \leq \mu_A + \mu_B  = m_1+m_2-\rho (k-2) \leq 15$.  Theorem~\ref{thm: 2star} implies that there is an equitable $L''$-coloring $h$ of $G''$; that is, $h$ is a proper $L''$-coloring of $G''$ that uses no color more than $\mu_A + 1$ times.  Now, we extend $h$ to a proper $L'$-coloring of $G'$ by coloring each $b_i \in  \{b_1, \ldots, b_d \}$ with an element in $L'(b_i) - \{h(u_0)\}$.  The proper $L'$-coloring that we obtain clearly uses no color more than $\mu_A+1+d = \mu_B+1$ times which immediately implies that it uses no color more than $\rho $ times.  Combining this proper $L'$-coloring with $g$ completes an equitable $L$-coloring of $G$.
\end{proof}

Next, we demonstrate that for $k \geq 2$, we can not drop the second inequality in the statement of Theorem~\ref{thm: Main}.

\begin{pro} \label{pro: tconstruct}
Suppose $k \geq 2$.  Then, $K_{1, (k-1)(k^3-k+2)} + K_{1, k^3}$ is not equitably $k$-choosable.
\end{pro}

Before we begin the proof, notice that if $G = K_{1, (k-1)(k^3-k+2)} + K_{1, k^3}$, then
$$\rho(G,k) = \left \lceil \frac{2+(k-1)(k^3-k+2) + k^3}{k} \right \rceil = k^3-k+3.$$
Also, $(k-1)(k^3-k+2) = (k-1)(\rho-1) = \rho(k-1) + 1-k \leq \rho(k-1) - 1$ which means that $G$ satisfies the first inequality in Theorem~\ref{thm: Main}. But for the second inequality, $m_1+m_2 - \rho(k-2) = 2(k^3-k+2)\ge 16$.

\begin{proof}
Let $G = K_{1, (k-1)(k^3-k+2)} + K_{1, k^3}$.  Suppose $G_1$ and $G_2$ are the components of $G$.  Moreover, suppose $G_1$ has bipartition $\{u_0\}$, $A= \{u_i : i \in [(k-1)(k^3-k+2)] \}$, and suppose $G_2$ has bipartition $\{w_0\}$, $B= \{w_i : i \in [k^3] \}$.  We will now construct a $k$-assignment $L$ for $G$ with the property that there is no equitable $L$-coloring of $G$.
	
For each $v \in V(G_1)$, let $L(v)= [k]$.  Also, let $L(w_0) = \{k+1, k+2, \ldots, 2k \}$.  Now, let $O = \{O_1, \ldots, O_k \}$ be the set of all $(k-1)$-element subsets of $[k]$.  Then, let $P = \{ \{k+i \} \cup O_j : i \in [k], j \in [k] \}$.  Clearly, $|P|=k^2$.  So, we can name the elements of $P$ so that $P = \{P_1, \ldots, P_{k^2} \}$.  Finally, for each $i \in [k^2]$ and $j \in [k]$, let $L(w_{(i-1)k+j})=P_i$.
	
Now, for the sake of contradiction, suppose that $f$ is an equitable $L$-coloring of $G$.  We know that $f$ uses no color more than $\rho = k^3-k+3$ times.  Without loss of generality, suppose that $f(u_0)=1$.  Then, for each $i \in \{2, \ldots, k \}$, let $a_i = |f^{-1}(i) \cap A|$.  Clearly, $\sum_{i=2}^k a_i = (k-1)(k^3-k+2) = (k-1)(\rho-1)$.  Now, suppose that $f(w_0)=d$, and without loss of generality assume that $w_1, \ldots, w_k$ are the $k$ vertices in $B$ that were assigned the list $\{d \} \cup \{2, \ldots, k \}$ by $L$.  Then, for each $i \in \{2, \ldots, k \}$, let $b_i = |f^{-1}(i) \cap \{w_j : j \in [k]\}|$.  Since $f(w_j) \neq d$ for each $j \in [k]$, we have that $\sum_{i=2}^k b_i = k$.  We also have that $a_i + b_i \leq |f^{-1}(i)| \leq \rho$ for each $i \in \{2, \ldots, k\}$.  So, we see that
$$ (k-1)\rho \geq \sum_{i=2}^k (a_i + b_i) = \sum_{i=2}^k a_i + \sum_{i=2}^k b_i = (k-1)(\rho-1) + k = (k-1)\rho + 1$$
which is a contradiction.
\end{proof}

Finally, we will show that the converse of Theorem~\ref{thm: Main} does not hold.

\begin{pro}\label{pro: counter}
$K_{1,8} + K_{1,9(k-1)-1}$ is equitably $k$-choosable for all $k \geq 3$.
\end{pro}
Notice that $8 + 9(k-1)-1 > 15 +9(k-2)$.  So, this graph does not satisfy the second inequality in Theorem~\ref{thm: Main}. The proof illustrates how ideas from the proof of Theorem~\ref{thm: Main} can be applied even in this situation.

\begin{proof}
Let $G = K_{1,8} + K_{1,9(k-1)-1}$, and let the components of $G$ be $G_1$ and $G_2$. Suppose the bipartition of $G_1$ is $\{w_0 \}$, $A= \{w_1, \ldots, w_{8} \}$ and the bipartition of $G_2$ is $\{u_0 \}$, $B=\{u_1, \ldots, u_{9(k-1)-1} \}$. Let $L$ be an arbitrary $k$-assignment for $G$. Note that $\rho(G,k) = 9$. For the sake of contradiction, suppose that $G$ is not equitably $L$-colorable. Let $S$ be the set containing all colors that appear in at least $9$ of the lists associated with the vertices in $A \cup B$. Suppose we run the $\epsilon$-greedy process on $G$ and $L$.

\emph{Observation 1}: $|S| \geq k-2$. Suppose that $|S| < k-2$. Note that $|\text{Ran}(g_\epsilon)| < k-2$. By Lemma~\ref{lem: thing2} we know that there is no color that appears in at least 9 of the lists assigned by $L_\epsilon$ to the vertices in $V(G_\epsilon)-\{u_0,w_0\}$. Therefore, by Lemma~\ref{lem: thing1} we know that $G$ is equitably $L$-colorable which is a contradiction.

For each $(k-2)$-element subset $P$ of $S$ let $\mathcal{C}_P$ denote the set of all partial $L$-colorings $f: D \rightarrow \mathcal{L}$ of $G$ such that $D \subset A \cup B$, $|D| = 9(k-2)$, $f(D) = P$, and each color class associated with $f$ is of size $9$. For each $f \in \mathcal{C}_P$ let $U_A^f$ (resp., $U_B^f$) be the set of vertices in $A$ (resp., $B$) not colored by $f$. Note that $U_A^f$ and $U_B^f$ are dependent on the choice of $P$.

\emph{Observation 2}: \emph{$|\Ran(g_\epsilon)| = k-2$ and $\mathcal{C}_{\Ran(g_\epsilon)} \neq \emptyset$}. This is easy to verify by assuming that  $|\text{Ran}(g_\epsilon)| < k-2$ and proceeding as we did in the proof of Observation 1.

Let $\mathcal{S}$ be the set containing all sets $P$ that are $(k-2)$-element subsets of $S$ and satisfy $\mathcal{C}_P \neq \emptyset$.  Observation 2 implies that $\mathcal{S}$ is nonempty.  Also, let $S' = \bigcup_{P \in \mathcal{S}} P $. For each $P \in \mathcal{S}$, let $g_P: D_P \rightarrow \mathcal{L}$ be a function chosen from the elements of $\mathcal{C}_P$ so that $||U_A^{g_P}| -|U_B^{g_P}||$ is as small as possible. We will write $g$ instead of $g_P$ when $P$ is clear from context. Let $G_P = G -D_P$. Note that it is possible for $|U_A^g| = 0$. Also note that $G_P$ is a copy of $K_{1,|U_A^g|} + K_{1,|U_B^g|}$. Let $L_P(v) = L(v) - P$ for all $v \in V(G_P)$.

\emph{Observation 3}: \emph{If $c \notin S'$ then $c$ is not in 9 of the lists associated with the vertices in $U_A^{g_P} \cup U_B^{g_P}$ by $L_P$ for each $P \in \mathcal{S}$}. Suppose that $ c \notin S'$ and $c$ is in 9 of the lists associated with the vertices in $U_A^{g_T} \cup U_B^{g_T}$ by $L_T$ for some $T \in \mathcal{S}$. Also suppose that $t_1 \in T$. We modify $g_T$ as follows: we uncolor the vertices that are colored with $t_1$ and color 9 of the vertices in $U_A^{g_T} \cup U_B^{g_T}$ that have $c$ in their original lists with $c$. Let $T' = (T \cup \{c\}) - \{t_1\}$. Clearly we see that this new partial coloring of $G$ is in $\mathcal{C}_{T'}$. Therefore it must be that $c \in S'$ which is a contradiction.

\emph{Observation 4}: $|S'| \geq k-1$. Suppose that $|S'|< k-1$.  This implies $|S'| = k-2$ which implies that $|\mathcal{S}| = 1$, and let $P$ be the element in $\mathcal{S}$. By Observation~3, there is no color that appears in at least $9$ of lists associated with the vertices in $U_A^g\cup U_B^g$ by $L_P$. Also note that $|L_P(v)| \geq 2$ for all $v \in V(G_P)$, and $G_P$ is 2-choosable. So, we know that $G_P$ is equitably $L_P$-colorable. Such a coloring of $G_P$ combined with $g_P$ completes an equitable $L$-coloring of $G$ which is a contradiction.

\emph{Observation 5}: \emph{For all $P \in \mathcal{S}$, $G_P = K_{1,8}+K_{1,8}$}. Suppose that there exists a $P \in \mathcal{S}$ such that $G_P \neq  K_{1,8}+K_{1,8}$. Since $||U_A^g| -|U_B^g||$ is as small as possible, we know that $|L_P(v)| \geq 3$ for all $v \in U_B^g$ (by an argument similar to that of case two in Theorem~\ref{thm: Main}). Note that $|U_B^g|-|U_A^g|  \geq 2$. In the case that there is no color that appears in at least $9$ of lists associated by $L_P$ with the vertices in $U_A^g\cup U_B^g$ we can complete an equitable $L$-coloring of $G$ as we did in Observation 4. Otherwise by Lemma~\ref{lem: reduce} we know that there exists a proper $L_P$-coloring of $G_P$ that uses a color no more than 9 times. Such a coloring of $G_P$ combined with $g_P$ completes an equitable $L$-coloring of $G$ which is a contradiction.

\emph{Observation 6}: \emph{For all $P \in \mathcal{S}$, $|L_P(v)| = 2$ for all $v \in U_A^g \cup U_B^g$. Consequently, $P \subseteq L(v)$ for all $v \in U_A^g \cup U_B^g$.} Suppose that for some $P \in \mathcal{S}$ there exists a $v' \in U_A^g \cup U_B^g$ such that $|L_P(v')| \geq 3$. Without loss of generality we suppose that $v' \in U_A^g$ (this is permissible by Observation~5). Let $G_P' = G_P-\{v'\}$, and note that $G'_P$ is a copy of $K_{1,7} + K_{1,8}$. Also let $L'_P(v) = L_P(v)$ for all $v \in V(G'_P)$. We arbitrarily remove colors from $L'_P(v)$ until $|L'_P(v)|=2$ for all $v \in V(G'_P)$. We know by Theorem~\ref{thm: 2star} that $G'_P$ is equitably 2-choosable which implies that there exists an equitable $L'_P$-coloring $h$ of $G'_P$. Note that there can exist at most one color $c \in h(V(G'_P))$ such that $|h^{-1}(c)| = 9$. If there is such a color remove it from $L_P(v')$, and also remove $h(w_0)$ from $L_P(v')$ if $h(w_0) \in L_P(v')$.  Coloring $v'$ with a color still in $L_P(v')$ completes an equitable $L$-coloring of $G$.

\emph{Observation 7}: \emph{For all $P \in \mathcal{S}$, $P \subseteq L(v)$ for each $v \in A \cup B$.  Consequently, $S' \subseteq L(v)$ for each $v \in A \cup B$.}  Suppose $P \in \mathcal{S}$.  If $v \in A$, it is clear that $P \subseteq L(v)$ by Observation~6 since Observation~5 implies $A = U_A^g$.  So, suppose for the sake of contradiction that $v' \in A \cup B$ has the property that $P$ is not a subset of $L(v')$.  We know that $v' \in B$, and Observation~6 implies that $v' \in B - U_B^g$.  So, $v' \in D_P$, and $g(v') \in P$.  Now, modify $g$ as follows.  Color an element $w \in U_B^g$ with $g(v')$, and remove the color $g(v')$ from $v'$.  We know the resulting coloring is still a partial $L$-coloring of $G$ by Observation~6.  Let $G'$ be the subgraph of $G$ induced by the vertices of $G$ not colored by this partial $L$-coloring.  Notice $G' = K_{1,8} + K_{1,8}$.  Let $L'(v) = L(v)-P$ for each $v \in V(G')$.  Clearly, $|L'(v')| \geq 3$.  So, we can complete an equitable $L$-coloring of $G$ by following the argument in Observation~6.  This however is a contradiction.

We note that Observation 7 implies that $|S'| \leq k$.

\emph{Observation 8}: \emph{$|S'| \neq k$.  Consequently, $|S'|=k-1$.} Suppose that $|S'| =k$, and $S' = \{c_1,c_2,c_3, \ldots , c_k\}$. By Observation~7 we know that $L(v) = \{c_1,c_2,c_3,\ldots , c_k\}$ for all $v \in A \cup B$. Color $G$ as follows:
$$h(v) = \begin{cases}
c_i \text{ if } v \in \{u_j:1 + 9(i-1) \leq j \leq 9i\} \text{ where } i \in [k-2]\\
c_{k-1} \text{ if } v \in A\\
c_k \text{ if } v \in \{u_j: 1+9(k-2) \leq j \leq 9(k-1)-1\}\\
c' \text{ if } v = w_0\\
c'' \text{ if } v = u_0
\end{cases}$$ where $c' \in L(w_0)-\{c_1, \ldots, c_{k-1}\}$ and $c'' \in L(u_0)-\{c_1, \ldots, c_{k-2}, c_k\}$. Notice that $h$ is an equitable $L$-coloring of $G$ which is a contradiction.

Now we will complete the proof. By Observation~8 we may suppose that $S' = \{c_1,c_2,\ldots , c_{k-1}\}$.  We know that either: (1) $(L(u_0) \cup L(w_0)) \cap S' \neq \emptyset$ or (2) $(L(u_0) \cup L(w_0)) \cap S' = \emptyset$.  We handle the first case by considering sub-cases where $L(u_0)$ contains an element of $S'$ and where $L(w_0)$ contains an element of $S'$.  First, without loss of generality suppose $c_{k-1} \in L(u_0)$, and color $G$ according to the function $h$ defined as follows:
$$h(v) = \begin{cases}
c_i \text{ if } v \in \{u_j:1 + 9(i-1) \leq j \leq 9i\} \text{ where } i \in [k-2]\\
c_{k-1} \text{ if } v \in A \cup\{u_0\}\\
d_j \text{ if } v \in \{u_j: 9(k-2)+1 \leq j \leq 9(k-1)-1\}\\
c' \text{ if } v = w_0
\end{cases}$$ where $c' \in L(w_0)-\{c_1, \ldots, c_{k-1}\}$ and $d_j \in L(u_j)-\{c_1, \ldots, c_{k-1}\}$ for each $9(k-2)+1 \leq j \leq 9(k-1)-1$.  Clearly, $h$ is an equitable $L$-coloring of $G$ which is a contradiction.  Second, without loss of generality suppose $c_{k-1} \in L(w_0)$, and color $G$ according to $h$ defined as follows:
$$h(v) = \begin{cases}
c_i \text{ if } v \in \{u_j:1 + 9(i-1) \leq j \leq 9i\} \text{ where } i \in [k-2]\\
c_{k-1} \text{ if } v \in \{u_j: 9(k-2)+1 \leq j \leq 9(k-1)-1\} \cup\{w_0\}\\
d_j \text{ if } v \in A\\
c' \text{ if } v = u_0
\end{cases}$$ where $c' \in L(u_0)-\{c_1, \ldots, c_{k-1}\}$ and $d_j \in L(w_j)-\{c_1, \ldots, c_{k-1} \}$ for each $j \in [8]$. Clearly $h$ is an equitable $L$-coloring of $G$ which is a contradiction.

In the second case, suppose that $L(u_0) = \{c'_1,c'_2,c'_3, \ldots , c'_k\}$ and $L(w_0) = \{c''_1,c''_2,c''_3, \ldots , c''_k\}$. Without loss of generality assume $P = \{c_2,c_3,\ldots, c_{k-1}\} \in \mathcal{S}$. We begin by coloring vertices in $A \cup B$ according to $g_P$.  By Observation~5, we know that $U_A^g = A$.  Note that $$\sum _{i \in [k]}|L_P^{-1} ( \{c_1,c'_i\}) \cap U_B^g| \leq 8 \text{ and } \sum _{i \in [k]} |L_P^{-1} ( \{c_1,c''_i\}) \cap A| \leq 8.$$ So without loss of generality assume $|L_P^{-1}(\{c_1,c'_1\}) \cap U_B^g| \leq \lfloor 8 / k \rfloor$ and $|L_P^{-1}(\{c_1,c''_1\}) \cap A| \leq \lfloor 8 / k \rfloor$. Color all vertices in $(L_P^{-1}(\{c_1,c'_1\}) \cap U_B^g) \cup (L_P^{-1}(\{c_1,c''_1\}) \cap A)$ with $c_1$, color $u_0$ with $c'_1$, and color $w_0$ with $c''_1$. Note that we used $c_1$ at most $2\lfloor 8 / k \rfloor$ times which is clearly less than 9. So we arbitrarily color uncolored vertices with $c_1$ until exactly 9 vertices are colored with $c_1$ (this is possible by Observation~7). Let $U$ be the set containing all uncolored vertices in $A \cup U_B^g$. Let $$L_P'(v) = \begin{cases}
L_P(v) - \{c_1,c'_1\} \text{ if } v \in (U_B^g \cap U)\\
L_P(v) - \{c_1,c''_1\} \text{ if } v \in (A \cap U)
\end{cases}.$$ Note that $|U| = 7$, and $|L'_P(v)| \geq 1$ for each $v \in U$.  So, we can color each $v \in U$ with a color in $L'_P(v)$. This completes an equitable $L$-coloring of $G$ which is a contradiction.
\end{proof}

\vspace{5mm}

\noindent \textbf{Acknowledgment:}  The authors would like to thank Marcus Schaefer for his helpful comments on this paper.  The authors would also like to thank Paul Shin for helpful conversations.

\end{document}